\documentclass{amsart}

\usepackage{amsmath,amsthm,amssymb}
\usepackage[all]{xy}

\theoremstyle{plain}

\newtheorem{definition}{Definition}
\newtheorem{thm}[definition]{Theorem}

\newtheorem{lem}[definition]{Lemma}
\newtheorem{cor}[definition]{Corollary}
\newtheorem{rei}[definition]{Example}

\newcommand{\Mat}{{\rm Mat}}
\newcommand{\Det}{{\rm Det}}
\newcommand{\Sdet}{{\rm Sdet}}
\newcommand{\Tr}{{\rm Tr}}
\newcommand{\sgn}{{\rm sgn}}
\newcommand{\Hom}{{\rm Hom}}

\begin{document}
\title[]{Factorizations of group determinant in group algebra for any abelian subgroup}
\author[N. Yamaguchi]{Naoya Yamaguchi}
\date{\today}

\maketitle

\begin{abstract}
We give a further extension and generalization of Dedekind's theorem over those presented by Yamaguchi. 
In addition, we give a corollary on irreducible representations of finite groups and a conjugation of the group algebra of the groups which have an index-two abelian subgroups. 
\end{abstract}

\section{\bf{INTRODUCTION}}
In this paper, we give an extension and generalization of Dedekind's theorem over those presented in reference \cite{Y}. 
The generalization in turn leads to a corollary on irreducible representations of finite groups. 
In addition, if a finite group has an index-two abelian subgroup, 
we can define a conjugation of elements of the group algebra by using the further extension of Dedekind's theorem.

Let $G$ be a finite group, 
$\widehat{G}$ a complete set of irreducible representations of $G$ over $\mathbb{C}$, and $\Theta(G)$ the group determinant of $G$. 
$\Theta(G)$ is the determinant of a matrix whose elements are independent variables $x_{g}$ corresponding to $g \in G$. 
Dedekind proved the following theorem about the irreducible factorization of the group determinant for any finite abelian group.

\begin{thm}[Dedekind \cite{W}]\label{thm:1.0.1}
Let $G$ be a finite abelian group. 
We have
$$
\Theta(G) = \prod_{\chi \in \widehat{G}} \sum_{g \in G} \chi (g) x_{g}. 
$$
\end{thm}

Frobenius gave a generalization of Dedekind's theorem. 
In particular, he proved the following theorem about the irreducible factorization of the group determinant for any finite group.

\begin{thm}[Frobenius \cite{C2}]\label{thm:1.0.2}
Let $G$ be a finite group, for which we have the irreducible factorization, 
$$
\Theta(G) = \prod_{\varphi \in \widehat{G}} \det{\left( \sum_{g \in G} \varphi (g) x_{g} \right)^{\deg{\varphi}}}. 
$$
\end{thm}

Let $\mathbb{C} G$ be the group algebra of $G$ over $\mathbb{C}$, 
$R = \mathbb{C}[x_{g}] = \mathbb{C}\left[x_{g} ; g \in G \right]$ the polynomial ring in $\left\{ x_{g} \: \vert \: g \in G \right\}$ with coefficients in $\mathbb{C}$, 
$RG = R \otimes \mathbb{C}G = \left\{ \sum_{g \in G} A_{g} g \: \vert \: A_{g} \in R \right\}$ the group algebra of $G$ over $R$, 
$H$ an abelian subgroup of $G$, and $[G : H]$ the index of $H$ in $G$. 
The paper \cite{Y} gives the following extension and generalization of Dedekind's theorem that are different from the theorem by Frobenius.

\begin{thm}[Extension of Dedekind's theorem \cite{Y}]\label{thm:1.0.3}
Let $G$ be a finite abelian group, $e$ the unit element of $G$, and $H$ a subgroup of $G$. 
For every $h \in H$, there exists a homogeneous polynomial $a_{h} \in R$ such that $\deg{a_{h}} = [G : H]$ and 
$$
\Theta(G) e = \prod_{\chi \in \widehat{H}} \sum_{h \in H} \chi(h) a_{h} h. 
$$
If $H = G$, we can take $a_{h} = x_{h}$ for each $h \in H$. 
\end{thm}

\begin{thm}[Generalization of Dedekind's theorem \cite{Y}]\label{thm:1.0.4}
Let $G$ be a finite abelian group and $H$ a subgroup of $G$. 
For every $h \in H$, there exists a homogeneous polynomial $a_{h} \in R$ such that $\deg{a_{h}} = \left[G : H \right]$ and 
$$
\Theta(G) = \prod_{\chi \in \widehat{H}} \sum_{h \in H} \chi(h) a_{h}. 
$$
If $H = G$, we can take $a_{h} = x_{h}$ for each $h \in H$. 
\end{thm}

Here, we give a further extension of Theorem~$\ref{thm:1.0.3}$ and generalization of Theorem~$\ref{thm:1.0.4}$.

\subsection{Results}
The following theorem is the further extension of Dedekind's theorem.

\begin{thm}[Further extension of Dedekind's theorem]\label{thm:1.1.1}
Let $G$ be a finite group, $e$ the unit element of $G$, and $H$ an abelian subgroup of $G$. 
For every $h \in H$, there exists a homogeneous polynomial $a_{h} \in R$ such that $\deg{a_{h}} = [G:H]$ and 
$$
\Theta(G) e = \prod_{\chi \in \widehat{H}} \sum_{h \in H} \chi(h) a_{h} h. 
$$
If $H$ is normal and $h$ is a conjugate of $h'$ on $G$, then $a_{h} = a_{h'}$. 
\end{thm}

Note that the equality in Theorem~$\ref{thm:1.1.1}$ is an equality in $R H$. 
Theorem~$\ref{thm:1.1.1}$ is proven using an extension of the group determinant $\Theta(G:H)$. 
The group determinant $\Theta(G:H)$ is an element of $RH$, and it is defined by using a left regular representation of $RH$. 
The left regular representation is reviewed in Section~$2$. 
In addition, Section~$2$ gives two expressions for the regular representation and shows that composition of regular representations is a regular representation. 
These expressions are helpful for describing some of the properties of $\Theta(G:H)$.

Above, we said that the group determinant is defined by using a left regular representation. 
In more detail, we define a noncommutative determinant by using a left regular representation and define the group determinant by using the noncommutative determinant. 
We know that the noncommutative determinant is analogous to the Study determinant. 
The Study determinant is a quaternionic determinant, defined by using the regular representation $\psi$ of the quaternions \cite{A}. 
In Section~$3$ and $4$, we describe the relationship between the noncommutative determinant and the Study determinant and their properties.

In the next section, we define the extension of the group determinant $\Theta(G:H)$ and give some properties of $\Theta(G:H)$.

In Section~$6$, we prove the further extension and generalization of Dedekind's theorem. 
In particular, Theorem~$\ref{thm:1.1.1}$ leads to the following theorem that is the further generalization of Dedekind's theorem.

\begin{thm}[Further generalization of Dedekind's theorem]\label{thm:1.1.2}
Let $G$ be a finite group and $H$ an abelian subgroup of $G$. 
For every $h \in H$, there exists a homogeneous polynomial $a_{h} \in R$ such that $\deg{a_{h}} = [G:H]$ and 
$$
\Theta(G) = \prod_{\chi \in \widehat{H}} \sum_{h \in H} \chi(h) a_{h}. 
$$
If $H$ is normal and $h$ is a conjugate of $h'$ on $G$, then $a_{h} = a_{h'}$. 
\end{thm}

From Theorem~$\ref{thm:1.1.2}$, we have the following corollary on irreducible representations of finite groups.

\begin{cor}\label{cor:1.1.3}
Let $G$ be a finite group and $H$ an abelian subgroup of $G$. 
For all $\varphi \in \widehat{G}$, we have 
\begin{align*}
\deg{\varphi} \leq [G:H]. 
\end{align*}
\end{cor}

In the last section, 
we define a conjugation of the group algebra of the group which has an index two abelian subgroup. 
The conjugation comes from the noncommutative determinant. 
By applying the conjugation, we arrive at an inverse formula of $2 \times 2$ matrix.

\section{\bf{Regular representation}}
Here, we describe the left regular representation of the group algebra and give two expressions for the representation. 
In addition, we show that a composition of regular representations is a regular representation.

Let $R$ be a commutative ring, 
$G$ a group, 
$H$ a subgroup of $G$ of finite index and $RG$ the group algebra of $G$ over $R$ whose elements are all possible finite sums of the form $\sum_{g \in G} a_{g} g$, where $a_{g} \in R$. 
We take a complete set $T = \{ t_{1}, t_{2}, \ldots, t_{[G:H]} \}$ of left coset representatives of $H$ in $G$, where $[G:H]$ is the index of $H$ in $G$.

\begin{definition}[Left regular representation]\label{def:2.1.1}
For all $A \in \Mat(m, RG)$, there exists a unique $L_{T}(A) \in \Mat(m [G:H], RH)$ such that 
$$
A (t_{1} I_{m} \: t_{2} I_{m} \: \cdots \: t_{[G:H]} I_{m}) = (t_{1} I_{m} \: t_{2} I_{m} \: \cdots \: t_{[G:H]} I_{m}) L_{T}(A). 
$$
We call the map $L_{T} : \Mat(m, R G) \ni A \mapsto L_{T}(A) \in \Mat(m [G:H], R H)$ the left regular representation from $\Mat(m, RG)$ to $\Mat(m [G:H], RH)$ with respect to $T$. 
\end{definition}

Obviously, $L_{T}$ is an injective $R$-algebra homomorphism. 

\begin{rei}\label{rei:2.1.2}
Let $G = \mathbb{Z}/ 2 \mathbb{Z} = \{ 0, 1 \}$ and 
$H = \{ 0 \}$ and $\alpha = x0 + y1 \in RG$. 
Then we have 
\begin{align*}
(x 0 + y 1) (0 \: 1) = (0 \: 1) 
\begin{bmatrix}
x 0 & y 0 \\ 
y 0 & x 0 
\end{bmatrix}. 
\end{align*}
\end{rei}

To give an expression for $L_{T}$, we define the map $\dot{\chi}$ by 
\begin{align*}
\dot{\chi}(g) = 
\begin{cases}
1 & g \in H, \\ 
0 & g \notin H 
\end{cases}
\end{align*}
for all $g \in G$ 
and we denote $(i, j)$ the $m \times m$ block element of an $(mn) \times (mn)$ matrix $M$ by $M_{(i, j)}$. 
We can now prove the following theorem.

\begin{lem}\label{lem:2.1.3}
Let $L_{T} : \Mat(m, RG) \rightarrow \Mat(m[G:H], RH)$ be the left regular representation with respect to $T$ and 
$A = \sum_{t \in T} t A_{t} \in \Mat(m, RG)$ where $A_{t} \in \Mat(m, RH)$. 
Then we have 
$$
L_{T}(A)_{(i, j)} = \sum_{t \in T} \dot{\chi}(t_{i}^{-1} t t_{j}) t_{i}^{-1} t A_{t} t_{j}. 
$$
\end{lem}
\begin{proof}
Let $r = [G:H]$. 
Then we have 
\begin{align*}
&( t_{1}I_{m} \: t_{2}I_{m} \: \cdots \: t_{r}I_{m}) \left( \sum_{t \in T} \dot{\chi} \left( t_{i}^{-1} t t_{j} \right) t_{i}^{-1} t A_{t} t_{j} \right)_{1 \leq i \leq r, 1 \leq j \leq r} \\ 
&= \left( 
\sum_{i = 1}^{r} \sum_{t \in T} \dot{\chi} ( t_{i}^{-1} t t_{1} ) t A_{t} t_{1} \: 
\sum_{i = 1}^{r} \sum_{t \in T} \dot{\chi} ( t_{i}^{-1} t t_{2} ) t A_{t} t_{2} \: 
\cdots \: 
\sum_{i = 1}^{r} \sum_{t \in T} \dot{\chi} ( t_{i}^{-1} t t_{m} ) t A_{t} t_{r} 
\right) \\ 
&= \left( \sum_{t \in T} t A_{t} \right) ( t_{1}I_{m} \: t_{2}I_{m} \: \cdots \: t_{r}I_{m}). 
\end{align*}
This completes the proof. 
\end{proof}

To get an another expression for $L_{T}$ when $H$ is a normal subgroup of $G$, 
we recall the Kronecker product. 
Let $A = (a_{i j})_{1 \leq i \leq m_{1}, 1 \leq j \leq n_{1}}$ be an $m_{1} \times n_{1}$ matrix and $B = (b_{i j})_{1 \leq i \leq m_{2}, 1 \leq j \leq n_{2}}$ be an $m_{2} \times n_{2}$ matrix. 
The Kronecker product $A \otimes B$ is the $(m_{1} m_{2}) \times (n_{1} n_{2})$ matrix, 
$$
A \otimes B = 
\begin{bmatrix} 
a_{11} B & a_{12} B & \cdots & a_{1 n_{1}} B \\ 
a_{21} B & a_{22} B & \cdots & a_{2 n_{1}} B \\ 
\vdots & \vdots & \ddots & \vdots \\ 
a_{m_{1} 1} B & a_{m_{1} 2} B & \cdots & a_{m_{1} n_{1}} B
\end{bmatrix}. 
$$

If $G = \{ g_{1}, g_{2}, \ldots, g_{|G|} \}$ is a finite group. 
Then the restriction of the left regular representation $L_{G} : G \rightarrow \Mat(|G|, R \{e \})$ with respect $G$ is 
$$
L_{G}(g)_{ij} = \dot{\chi}(g_{i}^{-1} g g_{j}) e 
$$
from Lemma~$\ref{lem:2.1.3}$. 
We often assume that $R\{e \} = R$, that $e = 1 \in R$. 
So, we can see that $L_{G}$ is a matrix form of the left regular representation of the group $G$.

Let \begin{align*}
P 
&= \begin{bmatrix} 
t_{1} I_{m} & & & \\ 
 & t_{2} I_{m} & & \\ 
 & & \ddots & \\ 
 & & & t_{[G:H]} I_{m} \\ 
\end{bmatrix}. 
\end{align*}
Thus, we have the follwing lemma.

\begin{lem}\label{lem:2.1.4}
Let $H$ be a normal subgroup of $G$, 
$L_{T}$ the left regular representation from $\Mat(m, RG)$ to $\Mat(m[G:H], RH)$ with respect to $T$, 
$L_{G/H}$ the left regular representation from $R(G/H)$ to $\Mat(|G/H|, R\{ eH \})$ with respect to $G/H$ and 
$A = \sum_{t \in T} t A_{t} \in \Mat(m, RG)$
where $A_{t} \in \Mat(m, RH)$. 
Accordingly, we have 
$$
L_{T}(A) = P^{-1} \left( \sum_{t \in T} L_{G/H}(tH) \otimes t A_{t} \right) P. 
$$
\end{lem}
\begin{proof}
From Lemma~$\ref{lem:2.1.3}$, we have 
\begin{align*}
\left( P^{-1} \left( \sum_{t \in T} L_{G/H} \otimes t A_{t} \right) P \right)_{(i, j)} 
&= t_{i}^{-1} I_{m} \left( \sum_{t \in T} \left(L_{G/H}(tH) \right)_{ij} t A_{t} \right) t_{j} I_{m} \\ 
&= \sum_{t \in T} \dot{\chi}(t_{i}^{-1} t t_{j}) t_{i}^{-1} t A_{t} t_{j} \\ 
&= L_{T}(A)_{(i, j)}. 
\end{align*}
This completes the proof. 
\end{proof}

We now show that a composition of regular representations is a regular representation. 
Theorem~$\ref{thm:6.1.5}$ requires the following lemma.

\begin{lem}\label{lem:2.1.5}
Let $K \subset H \subset G$ be a sequence of groups, 
$G = t_{1}H \cup t_{2}H \cup \cdots \cup t_{[G:H]}$, 
$H = u_{1}K \cup u_{2}K \cup \cdots \cup u_{[H:K]}$, 
$L_{T} : \Mat(m, RG) \rightarrow \Mat(m[G:H], RH)$ the representation with respect to $T$ and 
$L_{U} : \Mat(m[G:H], RH) \rightarrow \Mat(m[G:K], R K)$ the representation with respect to $U$. 
Then there exists a unique representation $L_{V}$ from $\Mat(m, RG)$ to $\Mat(m[G:K], R \{ e \})$ with respect to $V$ such that 
$$
L_{V} = L_{U} \circ L_{T}
$$
where $V = \{ v_{1}, v_{2}, \ldots, v_{[G:K]} \}$ is a complete set of the left coset representatives of $K$ in $G$ 
\end{lem}
\begin{proof}
Let $A \in \Mat(m, RG), r = [G:H]$ and $s = [H:K]$. 
By definition, we have 
\begin{align*}
A (t_{1}I_{m} \: t_{2}I_{m} \: \cdots \: t_{r}I_{m}) &= (t_{1}I_{m} \: t_{2}I_{m} \: \cdots \: t_{r}I_{m}) L_{T}(A), \\ 
L_{T}(A) (u_{1} I_{mr} \: u_{2} I_{mr} \: \cdots \: u_{s} I_{mr}) &= (u_{1} I_{mr} \: u_{2} I_{mr} \: \cdots \: u_{s} I_{mr}) L_{U}(L_{T}(A)). 
\end{align*}
Let $(a_{ij})_{1 \leq i \leq r, 1 \leq j \leq r} = L_{T}(A)$ and $(b_{ij})_{1 \leq i \leq s, 1 \leq j \leq s} = L_{U}(L_{T}(A))$ 
where $a_{ij} \in \Mat \left(m, RH \right)$ and $b_{ij} \in \Mat\left(mr, RH \right)$. 
Then we have 
$$
(a_{ij})_{1 \leq i \leq r, 1 \leq j \leq r} u_{p} I_{mr} = \sum_{q = 1}^{s} u_{q} b_{qp}. 
$$
We obtain 
$$
a_{ij} u_{p} = \sum_{q = 1}^{s} u_{q} (b_{qp})_{ij}. 
$$
Therefore, we have 
\begin{align*}
(A t_{i}) u_{j} 
&= \left( \sum_{p = 1}^{r} t_{p} a_{p i} \right) u_{j} \\ 
&= \sum_{p = 1}^{r} t_{p} (a_{pi} u_{j}) \\ 
&= \sum_{p = 1}^{r} t_{p} \left( \sum_{q = 1}^{s} u_{q} (b_{qj})_{pi} \right) \\ 
&= \sum_{p = 1}^{r} \sum_{q = 1}^{s} t_{p} u_{q} (b_{qj})_{pi}. 
\end{align*}
On the other hand, 
obviously $V = \left\{ t_{p} u_{q} \: \vert \: 1 \leq p \leq r, 1 \leq q \leq s \right\}$ is obviously a complete set of left coset representatives of $K$ in $G$. 
From $A t_{i} u_{j} = \sum_{p = 1}^{r} \sum_{q = 1}^{s} t_{p} u_{q} (b_{qj})_{pi}$, we have 
\begin{align*}
&A (t_{1}u_{1}I_{m} \: \cdots \: t_{r}u_{1}I_{m} \: t_{1}u_{2} I_{m} \: \cdots \: t_{r} u_{2} I_{m} \: \cdots \: t_{r} u_{s} ) \\ 
&\quad = (t_{1}u_{1}I_{m} \: \cdots \: t_{r}u_{1}I_{m} \: t_{1}u_{2} I_{m} \: \cdots \: t_{r} u_{2} I_{m} \: \cdots \: t_{r} u_{s} ) L_{U}(L_{T}(A)). 
\end{align*}
This completes the proof. 
\end{proof}

\section{\bf{Characteristics of image of representation when quotient group is abelian}}
In the case of $G/H$ is a finite abelian group, we have 
$$
L_{T} \left( \Mat(m, RG) \right) = \left\{ L_{T}(A) \: \vert \: A \in \Mat(m ,RG) \right\}. 
$$
We define $J_{t}$ by 
$$
J_{t} = P^{-1} \left( L_{G/H}(tH) \otimes I_{m} \right) P
$$
for all $t \in T$. 
The following lemma will be used to show that $B \in \Mat(m[G:H], RH)$ is an image of $L_{T}$ if and only if $B$ commutes with $J_{t}$.

\begin{lem}\label{lem:3.1.1}
Let $G/H$ be a finite abelian group and $L_{T}$ be the left regular representation from $\Mat(m, RG)$ to $\Mat(m[G:H], RH)$ with respect to $T$. 
Then, the elements of $L_{T}(\Mat(m, RG))$ and $J_{t}$ for all $t \in T$ are commutative. 
\end{lem}
\begin{proof}
Suppose $A = \sum_{t \in T} t A_{t} \in \Mat(m, RG)$ where $A_{t} \in \Mat(m, RH)$. 
From Lemma~$\ref{lem:2.1.4}$, we have $L_{T}(A) = P^{-1} \left( \sum_{t \in T} L_{G/H}(tH) \otimes tA_{t} \right)$. 
Therefore, we have 
\begin{align*}
L_{T}(A) J_{t'} 
&= P^{-1} \left( \sum_{t \in T} L_{G/H}(tH) \otimes tA_{t} \right) P P^{-1} \left( L(t'H) \otimes I_{m} \right) P \\ 
&= P^{-1} \left( \sum_{t \in T} L_{G/H}(tt'H) \otimes tA_{t} \right) P \\ 
&= P^{-1} \left( \sum_{t \in T} L_{G/H}(t'tH) \otimes tA_{t} \right) P \\ 
&= J_{t'} L_{T}(A)
\end{align*}
for all $t' \in T$. 
This completes the proof. 
\end{proof}

Now we are in a position to prove the following theorem.

\begin{thm}\label{thm:3.1.2}
Let $G/H$ be a finite abelian group and $L_{T}$ the left regular representation from $\Mat(m, RG)$ to $\Mat(m[G:H], RH)$ with respect to $T$. 
We have 
$$
L_{T}(\Mat(m, RG)) = \left\{ B \in \Mat(m[G:H], RH) \: \vert \: J_{t} B = B J_{t}, t \in T \right\}. 
$$
\end{thm}
\begin{proof}
From Lemma~$\ref{lem:3.1.1}$, we have 
$$
L_{T}(\Mat(m, RG)) \subset \left\{ B \in \Mat(m[G:H], RH) \: \vert \: J_{t} B = B J_{t}, t \in T \right\}. 
$$
We will show that 
$$
\left\{ B \in \Mat(m[G:H], RH) \: \vert \: J_{t} B = B J_{t}, t \in T \right\} \subset L_{T}(\Mat(m, RG)). 
$$
For all $B \in \Mat(m[G:H], RH)$, there exists $A \in \Mat(m, RG)$ and $B_{ij} \in \Mat(m, RH)$ such that 
\begin{align*}
B = L_{T}(A) + B'
\end{align*}
where 
$$
B' = \begin{bmatrix} 
0 & B_{12} & B_{13} & \cdots & B_{1 [G:H]} \\ 
0 & B_{22} & B_{23} & \cdots & B_{2 [G:H]} \\ 
\vdots & \vdots & \vdots & \ddots & \vdots \\ 
0 & B_{[G:H] 2} & B_{[G:H] 3} & \cdots & B_{[G:H] [G:H]} \\ 
\end{bmatrix}. 
$$
From Lemma~$\ref{lem:3.1.1}$, we have $B' J_{t} = J_{t} B'$. 
For all $p \in \{ 2, 3, \ldots, [G:H] \}$, there exists $t \in T$ such that ${J_{t}}_{(p, 1)} = t_{p} I_{m} t_{1}^{-1}$ and ${J_{t}}_{(i, 1)} = 0$ for all $i \neq p$. 
Therefore, we have 
\begin{align*}
B_{qp} t_{p} I_{m} t_{1}^{-1} 
&= (B' J_{t})_{(q, 1)} \\ 
&= (J_{t} B')_{(q, 1)} \\ 
&= 0
\end{align*}
for all $q \in \{ 1, 2, \ldots, [G:H] \}$. 
Thus, we have $B = L_{T}(A) \in L_{T}(\Mat(m, RG))$. 
This completes the proof. 
\end{proof}

Theorem~$\ref{thm:3.1.2}$ is similar to a property of a left regular representation of quaternion. 
Let $\mathbb{H}$ be the quaternions, $C + j D \in \Mat(m, \mathbb{H})$, where $C, D \in \Mat(m, \mathbb{C})$, and $\overline{C}$ the complex conjugation matrix of $C$. 
Then we have $(C + j D) (I_{m} \quad j I_{m}) = (I_{m} \quad j I_{m}) \psi(C + j D)$, where 
\begin{align*}
\psi(C + j D) = 
\begin{bmatrix} 
C & - \overline{D} \\ 
D & - \overline{C} 
\end{bmatrix}. 
\end{align*}
Hence, $\psi : \Mat(m, \mathbb{H}) \ni C + j D \mapsto \psi(C + j D) \in \Mat(2m, \mathbb{C})$ is a left regular representation. 
The following is known for the image of $\psi$ \cite{A}. 
$$
\psi(\Mat(m, \mathbb{H})) = \left\{ B \in \Mat(2m, \mathbb{C}) \: \vert \: JB = \overline{B} J \right\}
$$
where
$$
J = 
\begin{bmatrix}
0 & - I_{m} \\ 
I_{m} & 0 
\end{bmatrix}. 
$$

\section{\bf{Noncommutative determinant and some properties}}
In this section, 
we give a noncommutative determinant and describe its properties. 
This determinant is analogous to the Study determinant. 
Hence, we will define the determinant by using the regular representation of rings (group algebra).

Before defining the noncommutative determinant, 
we should remark that $A \in \Mat(m, RG)$ is invertible. 
Let $H$ be an abelian group.

\begin{lem}[Invertibility]\label{def:4.1.1}
For all $A, B \in \Mat(m, RG),\ AB = I_{m}$ if and only if $BA = I_{m}$. 
\end{lem}
\begin{proof}
Let $AB = I_{m}$. 
We have $L_{T}(A) L_{T}(B) = I_{m [G : H]}$. 
The elements of $L_{T}(A)$ and the element of $L_{T}(B)$ are elements of a commutative ring. 
Hence, $L_{T}(B) L_{T}(A) = I_{m [G : H]}$. 
Therefore, $L_{T}(BA - I_{m}) = 0$. 
Since $L_{T}$ is an injective, we have $BA = I_{m}$. 
\end{proof}

Therefore, we do not have to distinguish between the left and right inverses.

The noncommutative determinant is as follows.

\begin{definition}\label{def:4.1.2}
Let $H$ be an abelian subgroup of $G$ and $L$ be a left regular representation from $\Mat(m, RG)$ to $\Mat(m[G:H], RH)$. 
We define the map $\Det : \Mat(m, RG) \rightarrow RH$ by 
$$
\Det = \det{} \circ L. 
$$
\end{definition}

Let $T' = \{ t'_{1}, t'_{2}, \ldots, t'_{m} \}$ be an another complete set of left coset representatives of $H$ in $G$. 
Then, there exists $Q \in \Mat(m, RH)$ such that $L_{T} = Q^{-1} L_{T'} Q$. 
Therefore, we have 
\begin{align*}
\Det 
&= \det{} \circ L_{T} \\ 
&= \det{} \circ L_{T'}. 
\end{align*}
Thus, $\Det$ is an invariant under change of the left regular representation; hence, $\Det$ is well-defined.

$\Det$ has the following properties.

\begin{thm}\label{thm:4.1.3}
For all $A, B \in \Mat(m, RG)$, 
\begin{enumerate}
\item $\Det(AB) = \Det(A) \Det(B)$. 
\item $A \in \Mat(m, RG)$ is invertible if and only if $\Det(A) \in RH$ is invertible. 
\end{enumerate}
\end{thm}
\begin{proof}
$\Det$ is a homomorphism, because $L_{T}$ and $\det$ are homomorphisms. 
Therefore, the equation $(1)$ holds. 
Now let us prove $(2)$. 
If $A$ is invertible, there exists $B \in \Mat(m, RG)$ such that $A B = I_{m}$. 
Hence, $L_{T}(A) L_{T}(B) = I_{m [G : H]}, L_{T}(A)$ is invertible. 
If $\Det(A)$ is invertible, there exists $B \in \Mat(m [G : H], RH)$ such that $L_{T}(A) B = I_{m [G:H]}$. 
Therefore, 
$$
A (t_{1} I_{m} \: t_{2} I_{m} \: \cdots \: t_{[G:H]} I_{m}) B = (t_{1} I_{m} \: t_{2} I_{m} \: \cdots \: t_{[G:H]} I_{m}) I_{m [G : H]}. 
$$
Thus, $A$ is invertible. 
\end{proof}

Next let us define characteristic polynomial of $A \in \Mat(m, KG)$.

\begin{definition}[Characteristic polynomial]\label{def:4.1.4}
Let $H$ be an abelian group of $G$ and $L$ be a left regular representation from $\Mat(m, RG)$ to $\Mat(m[G:H], RH)$. 
For all $A \in \Mat(m, RG)$, we define $\Phi_{A}(X)$ by 
\begin{align*}
\Phi_{A}(X) 
&= \Det(X I_{m} - A) \\ 
&= \det(X I_{m [G:H]} - L_{T}(A)) 
\end{align*}
where $X$ is an independent variable such that $L_{T}(X B) = X L_{T}(B)$ and $\alpha X = X \alpha$ for any $B \in \Mat(m, RG)$ and $\alpha \in RH$. 
\end{definition}

We have the following lemma.

\begin{lem}\label{lem:4.1.5}
Let $H$ be a normal abelian subgroup of $G$ and $\Phi_{A}(X)$ the characteristic polynomial of $A$ over $RH$. 
Then we have $\Phi_{g^{-1} A g}(X) = \Phi_{A}(X)$ for all $g \in G$. 
\end{lem}
\begin{proof}
Since $f_{g} : G/H \ni t_{i} H \mapsto gt_{i} H \in G/H$ is a bijection for all $g \in G$, 
for all $g \in G$, there exists $P \in \Mat(m [G:H], RH)$ such that 
$$
g (t_{1} I_{m} \: t_{2} I_{m} \: \cdots \: t_{[G:H]} I_{m}) = (t_{1} I_{m} \: t_{2} I_{m} \: \cdots \: t_{[G:H]} I_{m}) P. 
$$
Therefore, we have 
\begin{align*}
\Phi_{g^{-1}Ag}(X) 
&= \det{(X I_{m [G:H]} - L_{T}(g^{-1} A g))} \\ 
&= \det{(X I_{m [G:H]} - P^{-1} L_{T}(A) P)} \\ 
&= \Phi_{A}(X). 
\end{align*}
Here, we should remark that $P^{-1} L_{T}(A) P \in \Mat(m[G:H], RH)$, since $H$ is a normal subgroup of $G$. 
This completes the proof. 
\end{proof}

We denote the center of the ring $R$ by $Z(R)$. 
The following corollary will be used in the proof of Theorem~$\ref{thm:6.1.5}$.

\begin{cor}\label{cor:4.1.6}
Let $H$ be a normal abelian subgroup of $G$ and 
$$
\Phi_{A}(X) = X^{m [G:H]} + a_{(m-1)[G:H]} X^{(m-1)[G:H]} + \cdots + a_{0}
$$
the characteristic polynomial of $A$ over $RH$. 
Then we have $a_{i} \in Z(RG) \cap RH$ for all $0 \leq i \leq (m-1)[G:H]$. 
In particular, $a_{0} = \Det(A)$ and $a_{(m-1)[G:H]} = \Tr(L(A)) \in Z(RG) \cap RH$. 
\end{cor}

Next let us prove a Cayley-Hamilton type theorem for $\Phi_{A}(X)$.

\begin{thm}[Cayley-Hamilton type theorem]\label{thm:4.1.7}
Let 
$$
\Phi_{A}(X) = X^{m[G:H]} + a_{(m-1)[G:H]} X^{(m-1)[G:H]} + \cdots + a_{0}
$$
be the characteristic polynomial of $A$ over $RH$. 
We have 
\begin{align*}
\Phi_{A}(A) 
&= A^{m[G:H]} + a_{(m-1)[G:H]} A^{(m-1)[G:H]} + \cdots + a_{0} I_{m} \\ 
&= 0. 
\end{align*}
\end{thm}
\begin{proof}
From the Cayley-Hamilton theorem for commutative rings, 
$$
L_{T}(A)^{m[G:H]} + a_{(m-1)[G:H]} L_{T}(A)^{(m-1)[G:H]} + \cdots + a_{0} I_{m} = 0
$$
and $A (t_{1} I_{m} \: t_{2} I_{m} \: \cdots \: t_{[G:H]} I_{m}) = (t_{1} I_{m} \: t_{2} I_{m} \: \cdots \: t_{[G:H]} I_{m}) L_{T}(A)$, 
we have 
\begin{align*}
\Phi_{A}(A) (t_{1} I_{m} \: t_{2} I_{m} \: \cdots \: t_{[G:H]} I_{m}) 
&= (t_{1} I_{m} \: t_{2} I_{m} \: \cdots \: t_{[G:H]} I_{m}) 0 \\ 
&= (0 \quad 0 \quad \cdots \quad 0)
\end{align*}
Thus, we have $\Phi_{A}(A) = 0$. 
This completes the proof. 
\end{proof}

The noncommutative determinant $\Det{}$ is analogous to the Study determinant. 
Therefore, these determinant have similar properties.

The Study determinant $\Sdet{}$ is defined by $\det{} \circ \psi : \Mat(m, \mathbb{H}) \rightarrow \mathbb{C}$. 
The Study determinant has the following properties \cite{A}. 
For all $A, B \in \Mat(m, \mathbb{H})$, 
\begin{enumerate}
\item $\Sdet{AB} = \Sdet{A} \Sdet{B}$. 
\item $A \in \Mat(m, \mathbb{H})$ is invertible if and only if $\Sdet{A} \neq 0$. 
\item $\Sdet{A} \in \mathbb{R}$. Hence, $\Sdet{A}$ is a central element of $\mathbb{H}$. 
\end{enumerate}
That is, Theorem~$\ref{thm:4.1.3}$ and Corollary~$\ref{cor:4.1.6}$ are similar to above properties.

\section{\bf{Extension of the group determinant in the group algebra for any abelian subgroup}}
Here, we extend the group determinant in the group algebra for any subgroup and show that the extension determines invertibility in $\Mat(m, RG)$. 
First, let us recall the group determinant.

Let $G$ be a finite group, 
$\{ x_{g} \: \vert \: g \in G \}$ be independent commuting variables, and 
$R = \mathbb{C}[x_{g}] = \mathbb{C}[x_{g} ; g \in G]$ the polynomial ring in $\{ x_{g} \: \vert \: g \in G \}$ with coefficients in $\mathbb{C}$. 
The group determinant $\Theta(G) \in R$ is the determinant of a $|G| \times |G|$ matrix $\left( x_{g, h} \right)_{g, h \in G}$, where $x_{g, h} = x_{g h^{-1}}$ for $g, h \in G$, 
and is thus a homogeneous polynomial of degree $|G|$ in $x_{g}$.

Now let us extend the group determinant in the group algebra for any abelian subgroup.

\begin{definition}[Extension of the group determinant]\label{def:5.1.1}
Let $G$ be a finite group, 
$H$ an abelian subgroup of $G$, 
$\alpha = \sum_{g \in G} x_{g}g \in RG$ and $L : RG \rightarrow \Mat([G:H], RH)$ a left regular representation. 
We define 
$$
\Theta(G:H) = (\det{} \circ L)(\alpha). 
$$
We call $\Theta(G:H)$ an extension of the group determinant in the group algebra $RH$. 
\end{definition}

If $H = \{ e \}$, we know that $\Theta(G:H) = \Theta(G) e$. 
Thus, we can prove the following lemma.

\begin{lem}\label{lem:5.1.2}
Let $G$ be a finite group, 
$\Theta(G)$ the group determinant of $G$, 
$\alpha = \sum_{g \in G} x_{g} g \in RG$, and $L : RG \rightarrow \Mat(|G|, R\{ e \})$ a left regular representation. 
We have 
\begin{align*}
\Theta(G : \{ e \}) 
&= (\det{} \circ L)(\alpha) \\ 
&= \Theta(G) e. 
\end{align*}
\end{lem}
\begin{proof}
Let $L_{G}$ be the left regular representation from $RG$ to $\Mat(|G|, R \{ e \})$ with respect to $G$. 
From Lemma~$\ref{lem:2.1.3}$, we have 
\begin{align*}
L_{G} \left( \sum_{g \in G} x_{g} g \right)_{i j} 
&= \sum_{g \in G} \dot{\chi}(g_{i}^{-1} g g_{j}) x_{g} g_{i}^{-1} g g_{j} \\ 
&= 
\begin{cases}
x_{g} e & g_{i}^{-1} g g_{j} = e, \\ 
0 & g_{i}^{-1} g g_{j} \neq e. 
\end{cases}
\end{align*}
Therefore, we have 
$$
L_{T}(\alpha) = \left( x_{g_{i} g_{j}^{-1}} e \right)_{1 \leq i \leq |G|, 1 \leq j \leq |G|}. 
$$
This completes the proof. 
\end{proof}

Let us explain how the extension of the group determinant determines invertibility. 
Now the situation is that $x_{g}$ is an element of $R$ for any $g \in G$. 
Hence, we assume that $\sum_{g \in G} x_{g}g \in RG$ and $\Theta(G) = \det{(x_{g h^{-1}})_{g, h \in G}} \in R$. 
Accordingly, we get the following theorem from Theorem~$\ref{thm:4.1.3}$.

\begin{thm}\label{thm:5.1.3}
Let $\alpha = \sum_{g \in G} x_{g} g \in RG$. 
Then $\alpha$ is invertible if and only if $\Theta(G:H)$ is invertible. 
\end{thm}

Obviously, $\Theta(G:\{ e \}) = \Theta(G) e$ is invertible if and only if $\Theta(G) \neq 0$. 
Therefore, we get the following corollary.

\begin{cor}\label{cor:5.1.4}
Let $\alpha = \sum_{g \in G} x_{g} g \in RG$. 
Then $\alpha$ is invertible if and only if $\Theta(G) \neq 0$. 
\end{cor}

\section{\bf{Factorizations of the group determinant in the group algebra for any abelian subgroup}}
In this section, we give factorizations of the group determinant in the group algebra of abelian subgroups. 
The factorizations compose a further extension of Dedekind's theorem upon the one presented in \cite{Y}. 
This extension in turn leads to a further generalization of Dedekind's theorem. 
Moreover, the generalization leads to a corollary on irreducible representations of finite groups.

First, we give a number of lemmas that will be needed later. 
The following theorem is well known.

\begin{thm}\label{thm:6.1.1}
Let $G$ be a finite group, 
$\widehat{G} = \{ \varphi_{1}, \varphi_{2}, \ldots, \varphi_{s} \}$ a complete set of inequivalent irreducible representations of $G$, 
$d_{i} = \deg{\varphi_{i}}$, and $L_{G}$ the left regular representation of $G$. 
We have 
$$
L_{G} \sim d_{1} \varphi_{1} \oplus d_{2} \varphi_{2} \oplus \cdots \oplus d_{s} \varphi_{s}. 
$$
\end{thm}

Let $\chi \in \Hom(G, R)$. 
We define $F_{\chi}^{(m)} : \Mat(m, RG) \rightarrow \Mat(m, RG)$ by 
$$
F_{\chi} \left( \left( \sum_{g \in G} x_{ij}(g) g \right)_{1 \leq i \leq m, 1 \leq j \leq m} \right) = \left( \sum_{g \in G} \chi(g) x_{ij}(g) g \right)_{1 \leq i \leq m, 1 \leq j \leq m}
$$
where $x_{ij}(g) \in R$. 
Now we have the following lemmas.

\begin{lem}\label{lem:6.1.2}
Let $G$ be an abelian group, $\chi \in \Hom(G, R)$ and $A = \sum_{g \in G} A_{g}g \in \Mat(m, RG)$, where $A_{g} \in \Mat(m, R)$. 
If $\det{A} = \sum_{g \in G}a_{g} g$, where $a_{g} \in R$, 
 we have 
\begin{align*}
\det{\left( \sum_{g \in G} \chi(g) A_{g} g \right)} = \sum_{g \in G} \chi(g) a_{g} g. 
\end{align*}
Hence, we have 
$$
\det{} \circ F_{\chi}^{(m)} = F_{\chi}^{(1)} \circ \det{}. 
$$
\end{lem}
\begin{proof}
Let $A = \left( \sum_{g \in G} a_{ij}(g) g \right)_{1 \leq i \leq m, 1 \leq j \leq m}$, where $a_{ij}(g) \in R$. 
Then we have 
\begin{align*}
\det{A} 
&= \sum_{\sigma \in S_{m}} \sgn(\sigma) \left( \prod_{i = 1}^{m} \sum_{g \in G} a_{\sigma(i) i}(g) g \right) \\ 
\end{align*}
Therefore, we have 
\begin{align*}
\sum_{g \in G} \chi(g) a_{g} g 
&= F_{\chi}^{(1)}(\det{A}) \\ 
&= \sum_{\sigma \in S_{m}} \sgn(\sigma) \left( \prod_{i = 1}^{m} \sum_{g \in G} a_{\sigma(i) i}(g) \chi(g) g \right) \\
&= \det{\left( a_{ij}(g) \chi(g) g \right)} \\ 
&= \det{F_{\chi}^{(m)}(A)}. 
\end{align*}
This completes the proof. 
\end{proof}

\begin{lem}\label{lem:6.1.3}
Let $G$ be an abelian group, 
$H$ a normal subgroup of $G$, 
$L$ a left regular representation from $\Mat(m, RG)$ to $\Mat(m[G:H], RH)$, and 
$A = \sum_{t \in T} t A_{t}$ where $A_{t} \in \Mat(m, RH)$. 
We have 
$$
(\det{} \circ L) \left( \sum_{t \in T} t A_{t} \right) = \prod_{\chi \in \widehat{G/H}} \det{ \left( \sum_{t \in T} \chi(tH) t A_{t} \right) }. 
$$
\end{lem}
\begin{proof}
From Lemma~$\ref{lem:2.1.4}$ and Theorem~$\ref{thm:6.1.1}$, 
\begin{align*}
(\det{} \circ L) \left( \sum_{t \in T} t A_{t} \right) 
&= \det{ \left( P^{-1} \left( \sum_{t \in T} L_{G/H}(tH) \otimes t A_{t} \right) P \right) } \\ 
&= \det{ \left( \sum_{t \in T} L_{G/H}(tH) \otimes t A_{t} \right) } \\ 
&= \prod_{\chi \in \widehat{G/H}} \det{ \left( \sum_{t \in T} \chi(tH) \otimes t A_{t} \right) } \\ 
&= \prod_{\chi \in \widehat{G/H}} \det{ \left( \sum_{t \in T} \chi(tH) t A_{t} \right) }. 
\end{align*}
This completes the proof. 
\end{proof}

\begin{lem}\label{lem:6.1.4}
Let $G$ be an abelian group, 
$H$ a subgroup of $G$, 
$L_{1}$ a left regular representation from $\Mat(m, RG)$ to $\Mat(m[G:H], RH)$, and 
$L_{2}$ a left regular representation from $RG$ to $\Mat([G:H], RH)$. 
Then the following diagram is commutative. 
\[
\xymatrix@=2pt{
& RG \ar[rdd]^{L_{2}} & \\ 
& & & \\ 
\Mat(m, RG) \ar[dddd]^{L_{1}} \ar[ruu]^{\det{}} & & \Mat([G:H], RH) \ar[dddd]^{\det{}} \\ 
& & & & \\ 
& & & & \\ 
& & & & \\ 
& & & & \\ 
\Mat(m[G:H], RH) \ar[rr]^{\det{}} & & RH \\ 
}
\]
\end{lem}
\begin{proof}
Let $A = \sum_{t \in T} t A_{t}$ and $\det{A} = \sum_{t \in T} t a_{t}$, 
where $A_{t} \in \Mat(m, R)$ and $a_{t} \in R$. 
From Lemma~$\ref{lem:6.1.3}$, we have 
\begin{align*}
(\det{} \circ L_{1})(A) = \prod_{\chi \in \widehat{G/H}} \det{\left( \sum_{t \in T} \chi(tH) t A_{t} \right)}. 
\end{align*}
and 
\begin{align*}
\left( \det{} \circ L_{2} \circ \det{} \right)(A) 
&= \left( \det{} \circ L_{2} \right) \left( \sum_{t \in T} t a_{t} \right) \\ 
&= \prod_{\chi \in \widehat{G/H}} \left( \sum_{t \in T} \chi(tH) t a_{t} \right). 
\end{align*}
We regard $\chi : G/H \rightarrow R$ as $\chi : G \ni g \mapsto \chi(gH) \in R$. 
Accordingly, we have
\begin{align*}
\prod_{\chi \in \widehat{G/H}} \det{\left( \sum_{t \in T} \chi(tH) t A_{t} \right)} 
&= \prod_{\chi \in \widehat{G/H}} (\det{} \circ F_{\chi}^{(m)})(A) \\ 
&= \prod_{\chi \in \widehat{G/H}} (F_{\chi}^{(1)} \circ \det{}) (A) \\ 
&= \prod_{\chi \in \widehat{G/H}} F_{\chi}^{(1)} \left( \sum_{t \in T} t a_{t} \right) \\ 
&= \prod_{\chi \in \widehat{G/H}} \sum_{t \in T} \chi(tH) t a_{t} 
\end{align*} 
by Lemma~$\ref{lem:6.1.2}$. This completes the proof. 
\end{proof}

Now we are ready to state and prove the further extension of Dedekind's theorem.

\begin{thm}[Further extension of Dedekind's theorem]\label{thm:6.1.5}
Let $G$ be a finite group and $H$ be an abelian subgroup of $G$. 
For every $h \in H$, there exists a homogeneous polynomial $a_{h} \in R$ such that $\deg{a_{h}} = [G:H]$ and 
$$
\Theta(G) e = \prod_{\chi \in \widehat{H}} \sum_{h \in H} \chi(h) a_{h} h. 
$$
If $H$ is normal and $h$ is a conjugate of $h'$ on $G$, then $a_{h} = a_{h'}$. 
\end{thm}
\begin{proof}
From Lemma~$\ref{lem:2.1.5}$ and $\ref{lem:5.1.2}$, we have 
$$
(\det{} \circ L_{2} \circ L_{1}) \left( \sum_{g \in G} x_{g} g \right) = \Theta(G) e. 
$$
On the other hand, we have 
\begin{align*}
(\det{} \circ L_{3} \circ \det{} \circ L_{1}) \left( \sum_{g \in G} x_{g} g \right) 
&= (\det{} \circ L_{3}) \left( \sum_{h \in H} a_{h} h \right) \\ 
&= \prod_{\chi \in \widehat{H}} \left( \sum_{h \in H} \chi(h) a_{h} h \right)
\end{align*}
by Lemma~$\ref{lem:6.1.3}$. 
From Lemma~$\ref{lem:6.1.4}$, we can build the following commutative diagram. 
\[
\xymatrix@=2pt{
RG \ar[rrrdd]^{L_{1}} & & & & RH \ar[rdd]^{L_{3}} & \\ 
& & & & & & \\ 
& & & \Mat([G:H], RH) \ar[dddd]^{L_{2}} \ar[ruu]^{\det{}} & & \Mat(|H|, R\{e \}) \ar[dddd]^{\det{}} \\ 
& & & & & & & \\ 
& & & & & & & \\ 
& & & & & & & \\ 
& & & & & & & \\ 
& & & \Mat(|G|, R\{e \}) \ar[rr]^{\det{}} & & R\{e \} \\ 
}
\]
Therefore, we have 
$$
\Theta(G) e = \prod_{\chi \in \widehat{H}} \sum_{h \in H} \chi(h) a_{h} h. 
$$
If $H$ is a normal subgroup of $G$. 
From Corollary~$\ref{cor:4.1.6}$, we have 
$$
(\det{} \circ L_{1}) \left( \sum_{g \in G} x_{g} g \right) = \sum_{h \in H} a_{h} h \in Z(RG). 
$$
Hence, $a_{h} = a_{h'}$ when $h$ is a conjugate $h'$ on $G$. 
This completes the proof. 
\end{proof}

Now we are in a position to state and prove the further generalization of Dedekind's theorem. 
Let $F : RG \rightarrow R$ be the $R$-algebra homomorphism such that $F(g) = 1$ for all $g \in G$. 
We call the map $F$ the fundamental $RG$-function. 
The proof is as follows.

\begin{thm}[Further generalization of Dedekind's theorem]\label{thm:6.1.6}
Let $G$ be a finite group and $H$ an abelian subgroup of $G$. 
For every $h \in H$, there exists a homogeneous polynomial $a_{h} \in R$ such that $\deg{a_{h}} = [G:H]$ and 
$$
\Theta(G) = \prod_{\chi \in \widehat{H}} \sum_{h \in H} \chi(h) a_{h}. 
$$
If $H$ is normal and $h$ is a conjugate of $h'$ on $G$, then $a_{h} = a_{h'}$. 
\end{thm}
\begin{proof}
From Theorem~$\ref{thm:6.1.5}$ and the fundamental $RG$-function, we have 
$$
\Theta(G) = \prod_{\chi \in \widehat{H}} \sum_{h \in H} \chi(h) a_{h}. 
$$
This completes the proof. 
\end{proof}

From Theorem~$\ref{thm:1.0.2}$ and $\ref{thm:6.1.6}$, we have the following corollary.

\begin{cor}\label{cor:6.1.7}
Let $G$ be a finite group and $H$ an abelian subgroup of $G$. 
For all $\varphi \in \widehat{G}$, we have 
\begin{align*}
\deg{\varphi} \leq [G:H]. 
\end{align*}
\end{cor}

Remark that Corollary~$\ref{cor:6.1.7}$ follows from Frobenius reciprocity\cite{K}.

\section{{\bf Conjugate of the group algebra of the groups which have an index two abelian subgroup}}

In this section, we define a conjugation of elements of the group algebra of the group which has an index-two abelian subgroup.

First, we recall the conjugation of elements of $\mathbb{H}$. 
Let $z = x + j y \in \mathbb{H}$. 
The conjugation of $z$ defined as $\overline{z} = x - j y$. 
The conjugation has the following properties. 
\begin{enumerate}
\item $\Sdet{z} = z \overline{z}$. 
\item $\overline{\overline{z}} = z$. 
\item $z + \overline{z}$ and $z \overline{z} = \overline{z} z \in \mathbb{R} = Z(\mathbb{H})$. 
\item $\overline{z w} = \overline{w} \: \overline{z} \quad (w \in \mathbb{H})$. 
\item $z = \overline{z}$ if and only if $z \in \mathbb{H}$. 
\end{enumerate}
We refer to the above for definition of the conjugation of elements of the group algebra.

Let $R$ be a commutative ring, $G$ a group, 
$H$ an index-two abelian subgroup of $G$, 
$T = \{e, t \}$ a complete set of left coset representatives of $H$ in $G$, 
$L_{T}$ the left regular representation from $RG$ to $\Mat(2, RH)$ with respect to $T$, and 
$A = \alpha + t \beta \in RG$. 
We have 
\begin{align*}
L_{T}(A) = 
\begin{bmatrix}
\alpha & t \beta t \\ 
\beta & t^{-1} \alpha t 
\end{bmatrix} 
\end{align*}
and 
\begin{align*}
\Phi_{A}(X) &= 
\det{
\begin{bmatrix}
X - \alpha & t \beta t \\ 
\beta & X - t^{-1} \alpha t 
\end{bmatrix} 
} \\ 
&= X^{2} - (\alpha + t^{-1} \alpha t) X + \alpha t^{-1} \alpha t - \beta t \beta t. 
\end{align*}

We notice that 
\begin{align*}
&(X - (\alpha + t \beta)) (X - (t^{-1} \alpha t - t \beta)) \\ 
&\quad = X^{2} - (t^{-1} \alpha t - t \beta) X - (\alpha + t \beta) X + (\alpha + t \beta) (t^{-1} \alpha t - t \beta) \\ 
&\quad = X^{2} - (\alpha + t^{-1} \alpha t) X + \alpha t^{-1} \alpha t - \alpha t \beta + (t \beta t^{-1}) \alpha t - (t \beta t) \beta \\ 
&\quad = X^{2} - (\alpha + t^{-1} \alpha t) X + \alpha t^{-1} \alpha t - \alpha t \beta + \alpha (t \beta t^{-1}) t - \beta (t \beta t) \\ 
&\quad = X^{2} - (\alpha + t^{-1} \alpha t) X + \alpha t^{-1} \alpha t - \beta t \beta t \\ 
&\quad = \Phi_{A}(X). 
\end{align*}

Therefore, we define the conjugate of $A = \alpha + t \beta$ by 
$$
\overline{A} = t^{-1} \alpha t - t \beta. 
$$

The following theorem follows from Corollary~$\ref{cor:4.1.6}$ and a direct calculation:

\begin{thm}\label{thm:}
For all $A, B \in RG$, 
\begin{enumerate}
\item $\overline{\overline{A}} = A$. 
\item $A + \overline{A}$ and $A \overline{A} = \overline{A} A \in Z(RG)$. 
\item $\overline{AB} = \overline{B} \: \overline{A}$. 
\item $A = \overline{A}$ if and only if $A \in Z(RG)$. 
\end{enumerate}
\end{thm}

We give the inverse formula of $2 \time 2$ matrix by conjugation.

\begin{thm}\label{thm:5}
Let $A, B, C, D \in RG$. Then we have 
\begin{align*}
\begin{bmatrix}
A & B \\ 
C & D 
\end{bmatrix}^{-1} 
= 
\begin{bmatrix}
\overline{D} & \overline{B} \\ 
\overline{C} & \overline{D} 
\end{bmatrix}
(\alpha \overline{\alpha} - \beta \gamma)^{-1} 
\begin{bmatrix}
\overline{\alpha} & - \beta \\ 
- \gamma & \alpha 
\end{bmatrix}
\end{align*}
where 
\begin{align*}
\begin{bmatrix}
\alpha & \beta \\ 
\gamma & \overline{\alpha}
\end{bmatrix}
= 
\begin{bmatrix}
A \overline{D} + B \overline{C} & A \overline{B} + B \overline{A} \\ 
C \overline{D} + D \overline{C} & D \overline{A} + C \overline{B}
\end{bmatrix}. 
\end{align*}
\end{thm}
\begin{proof}
We have 
\begin{align*}
\begin{bmatrix}
\alpha & \beta \\ 
\gamma & \overline{\alpha}
\end{bmatrix}
= 
\begin{bmatrix} 
A & B \\ 
C & D
\end{bmatrix} 
\begin{bmatrix} 
\overline{D} & \overline{B} \\ 
\overline{C} & \overline{D} 
\end{bmatrix}. 
\end{align*}
From $\beta, \gamma \in Z(RG)$, these elements $\alpha, \overline{\alpha}, \beta, \gamma$ are interchangeable. 
Therefore, we can use the inverse formula for $2 \times 2$ matrix whose elements are in commutative ring. 
This completes the proof. 
\end{proof}

\clearpage

\thanks{\bf{Acknowledgments}}
I am deeply grateful to Prof. Hiroyuki Ochiai who provided the helpful comments and suggestions. 
Also, I would like to thank my colleagues in the Graduate School of Mathematics of Kyushu University, 
in particular Tomoyuki Tamura and Yuka Suzuki for comments and suggestions. 
I would also like to express my gratitude to my family for their moral support and warm encouragements. 
This work was supported by a grant from the Japan Society for the Promotion of Science (JSPS KAKENHI Grant Number 15J06842).

\medskip
\begin{flushleft}
Naoya Yamaguchi\\
Graduate School of Mathematics\\
Kyushu University\\
Nishi-ku, Fukuoka 819-0395 \\
Japan\\
n-yamaguchi@math.kyushu-u.ac.jp
\end{flushleft}

\end{document}